\documentclass[reqno]{amsart}
\usepackage{hyperref}
\usepackage[utf8]{inputenc}
\usepackage{amsmath}
\usepackage{amsfonts}
\usepackage{amssymb}
\usepackage{amsthm,color}
\usepackage{indentfirst}
\usepackage[left=2cm,right=2cm, top=2cm, bottom=2cm]{geometry}

\newcommand{\N}{\mathbb{N}}
\newcommand{\Nzero}{\mathbb{N}_0}
\begin{document}

\title[\hfilneg   ]{On the Gevrey regularity of the fifth-order Kadomtsev-Petviashvili-II equation: An improved approach}
	\date{}
	\author{Aissa Boukarou$^{*}$}
	
	\address{\textbf{Aissa Boukarou:}University of Science and Technology Houari Boumediene, Algiers, Algeria} 
	\email{boukarouaissa@gmail.com}	

        \author{Lamia Seghour}
	
	\address{\textbf{Lamia Seghour:} University of Science and Technology Houari Boumediene, Algiers, Algeria} 
	\email{seghour.lamia@gmail.com}

	\begin{abstract}
	In this paper, we improve and extend the results obtained by Boukarou et al. \cite{boukarou1} on the Gevrey regularity of solutions to a fifth-order Kadomtsev-Petviashvili-II equation. We establish Gevrey regularity in the time variable for solutions in $2+1$ dimensions, providing a sharper result obtained through a new analytical approach. Assuming that the initial data are Gevrey regular of order $\sigma \geq 1$ in the spatial variables, we prove that the corresponding solution is Gevrey regular of order $5 \sigma$ in time. Moreover, we show that the function $u(x, y, t)$, viewed as a function of $t$, does not belong to $G^z$ for any $1 \leq z<5 \sigma$.  Our proof introduces a new analytical method that establishes a general principle for dispersive equations of the form $ \partial_t u = \pm\partial_x^\alpha u + P(u),$ where $\partial_x^\alpha $ is the highest spatial derivative and $P(u)$ a polynomial in spatial derivatives of total order at most $\alpha-1$, the solution cannot belong to the Gevrey class $G^z$ in time for any $z$ satisfying $1 \leq z<\alpha \sigma$.
	\end{abstract}
	\maketitle
	\numberwithin{equation}{section}
\newtheorem{theorem}{Theorem}[section]
\newtheorem{lemma}[theorem]{Lemma}
\newtheorem{proposition}[theorem]{Proposition}
\newtheorem{remark}[theorem]{Remark}
\newtheorem{definition}[theorem]{Definition}
\allowdisplaybreaks	

	\section{Introduction}
The Gevrey classes on a domain $\Omega \subset\mathbb{R}^n $, introduced by Maurice Gevrey \cite{Gevrey1918}, that provide a fine gradation between the space of analytic functions $C^{\omega}(\Omega)$ and the space of smooth functions $C^{\infty}(\Omega)$. These spaces have become an indispensable tool in the study of partial differential equations, asymptotic analysis of solutions to evolution equations arising in physics. A function \( f \in C^\infty(\Omega) \) is said to belong to the Gevrey class $ G^\sigma(\Omega) $, with $ \sigma \ge 1 $, if there exist constants $ C_1, C_2 > 0$ such that for all $ \alpha \in \mathbb{N}^n $,
\[
|\partial^\alpha f(x)| \leq C_1 C_2^{|\alpha|}\, (|\alpha|!)^{\sigma}, \qquad x \in \Omega.
\]
If $\sigma = 1$, this estimate characterizes analytic functions $G^1(\Omega) =C^{\omega}(\Omega)$, while for $\sigma > 1$, the function is smoother than $C^\infty$ but not analytic. The smaller the index $\sigma$, the stronger the regularity.

The Kadomtsev-Petviashvili (KP) equation is a fundamental nonlinear dispersive model that describes the propagation of weakly nonlinear, long surface or plasma waves with weak transverse effects. It was first derived by B.B. Kadomtsev and V.I. Petviashvili as a two dimensional generalization of the well-known Korteweg-de~Vries equation \cite{Kadomtsev}. The equation of KP is given by
\[
\partial_{x}(\partial_{t}u + 6u\partial_{x}u + \partial_{x}^{3}u) + \varrho \partial_{y}^{2}u = 0,
\]
where $\varrho = \pm 1$ selects the dispersion type, defining the KP-I for $\varrho = -1$ and KP-II for $\varrho = 1$ variants. Physically, KP-I describes phenomena with negative dispersion, including capillary-gravity waves in shallow water, certain plasma waves, and internal waves in stratified fluids. Conversely, the KP-II equation is characteristic of systems exhibiting positive dispersion.

Boukarou et al. \cite{Boukarou2021, boukarou1, boukarou2} studied the Cauchy problem for the generalized Kadomtsev--Petviashvili I, the fifth-order Kadomtsev--Petviashvili I and Kadomtsev--Petviashvili II equations in analytic Bourgain spaces. They established local well-posedness for analytic initial data and proved that analyticity in both spatial and temporal variables is preserved as long as the solution exists. Furthermore, they obtained quantitative lower bounds on the radius of spatial analyticity, showing that it decays at most algebraically in time. They also studied the regularity with respect to $t, x$, and $y$, demonstrating that the solution is analytic in $x$ and $y$ and belongs to $G^{5 \sigma}$ in $t$. However, an open question remains concerning the optimal regularity in time within the Gevrey class $G^z$ for any $1 \leq z<5 \sigma$. The objective of the present paper is to improve and extend the results obtained by Boukarou et al.~\cite{boukarou1} on the Gevrey regularity of solutions to a fifth-order Kadomtsev Petviashvili  equation \[	\partial_{t}u+\alpha\partial_{x}^{3}u+\partial_{x}^{5}u+\partial^{-1}_{x}\partial^{2}_{y}u+u\partial_{x}u=0,\] by introducing a new analytical method. We establish Gevrey regularity in the time variable for solutions in $2+1$ dimensions, providing a sharper result achieved through this refined approach. Assuming that the initial data are Gevrey regular of order $\sigma \geq 1$ in the spatial variables, we prove that the corresponding solution is Gevrey regular of order $5 \sigma$ in time. Moreover, we show that the function $u(x, y, t)$, viewed as a function of $t$, does not belong to $G^z$ for any $1 \leq z<5 \sigma$.

\bigskip
The rest of the paper is organized as follows. Section 2: presents the functional setting and preliminary lemmas related to Gevrey spaces, as well as the statement of our main results. Section 3: is devoted to the main estimates and the proof of the persistence of Gevrey regularity. Finally, Section 4: provides the concluding remarks.

\section{Functional spaces and main results}
We consider a Cauchy problem for fifth-order Kadomtsev-Petviashvili II equation
\begin{equation} 
	\left\{
	\begin{array}{ll}\label{p1}
		\partial_{t}u+\alpha\partial_{x}^{3}u+\partial_{x}^{5}u+\partial^{-1}_{x}\partial^{2}_{y}u+u\partial_{x}u=0,  \\ \\
		u(x,y,0)=\varphi(x,y), 
	\end{array}
	\right.
\end{equation}
where $u=u(x,y,t)$ and $x,y,t, \alpha \in \mathbb{R}^{2}$. 

We begin by defining the function spaces needed in our analysis, starting with the spaces of Gevrey functions 
$G^{\delta,\sigma,s_{1},s_{2}}(\mathbb{R}^{2})$ that contain the initial data.

\begin{definition}[Gevrey space]
	For $s_{1}, s_{2} \in \mathbb{R}$, $\sigma \geq 1$, and $\delta > 0$, we define
	\begin{equation}
		G^{\delta,\sigma,s_{1},s_{2}}(\mathbb{R}^{2})
		= \left\lbrace 
		\varphi \in L^{2}(\mathbb{R}^{2}) \;:\; 
		\Vert \varphi \Vert_{G^{\delta,\sigma,s_{1},s_{2}}(\mathbb{R}^{2})} < \infty 
		\right\rbrace,
	\end{equation}
	where
	\begin{equation*}
		\Vert \varphi \Vert^{2}_{G^{\delta,\sigma,s_{1},s_{2}}(\mathbb{R}^{2})}
		= \int_{\mathbb{R}^{2}} 
		e^{2\delta(|\xi|^{1/\sigma} + |\mu|^{1/\sigma})} 
		\langle \xi \rangle^{2s_{1}} \langle \mu \rangle^{2s_{2}}
		|\widehat{\varphi}(\xi,\mu)|^{2} \, d\xi \, d\mu.
	\end{equation*}
	Here, $\langle \cdot \rangle = (1 + |\cdot|^{2})^{1/2}$ and $\widehat{\varphi}$ is the space Fourier transform of $\varphi$ that is defined as
	$$
	\widehat{\varphi}(\xi,\mu)=\int_{\mathbb{R}^2} e^{-i(x \xi+y \mu)} \varphi(x, y) d x d y .
	$$
\end{definition}

\begin{remark}
	For $\delta = 0$, the space $G^{0,\sigma,s_{1},s_{2}}(\mathbb{R}^{2})$ coincides with the anisotropic Sobolev space 
	$H^{s_{1},s_{2}}(\mathbb{R}^{2})$, defined by
	\begin{equation}
		H^{s_{1},s_{2}}(\mathbb{R}^{2})
		= \left\lbrace 
		\varphi \in L^{2}(\mathbb{R}^{2}) \;:\; 
		\Vert \varphi \Vert_{H^{s_{1},s_{2}}(\mathbb{R}^{2})} < \infty 
		\right\rbrace,
	\end{equation}
	with
	\begin{equation*}
		\Vert \varphi \Vert^{2}_{H^{s_{1},s_{2}}(\mathbb{R}^{2})}
		= \int_{\mathbb{R}^{2}}  
		\langle \xi \rangle^{2s_{1}} \langle \mu \rangle^{2s_{2}} 
		|\widehat{\varphi}(\xi,\mu)|^{2} \, d\xi \, d\mu.
	\end{equation*}
\end{remark}

\begin{definition}[Gevrey Bourgain space]
	Let $s_{1}, s_{2}, b \in \mathbb{R}$, $\sigma \geq 1$, and $\delta > 0$. 
	The analytic Gevrey--Bourgain space associated with the fifth-order KP-II equation is defined as the completion of the Schwartz space 
	$\mathcal{S}(\mathbb{R}^{3})$ with respect to the norm
	\begin{equation} \label{e1.6}
		\Vert u \Vert_{X^{s_{1},s_{2}}_{\delta,\sigma,b}(\mathbb{R}^{3})}
		= \left( 
		\int_{\mathbb{R}^{3}} 
		e^{2\delta(|\xi|^{1/\sigma} + |\mu|^{1/\sigma})}
		\langle \xi \rangle^{2s_{1}} \langle \mu \rangle^{2s_{2}} 
		\langle \tau - \phi(\xi,\mu) \rangle^{2b} 
		|\widehat{u}(\xi,\mu,\tau)|^{2} 
		\, d\xi \, d\mu \, d\tau
		\right)^{1/2}.
	\end{equation}
	
	where	$\phi(\xi,\mu)=\xi^{5}-\alpha\xi^{3}+\dfrac{\mu^{2}}{\xi}$ and $\widehat{u}$ is the space time
	Fourier transform of $u$ that is defined as
	$$
	\widehat{u}(\xi,\mu,\tau)=\int_{\mathbb{R}^3} e^{-i(x \xi+y \mu+t\tau)} \varphi(x, y,t) d x d y dt .
	$$

	For $\delta = 0$, the space $X^{s_{1},s_{2}}_{0,\sigma,b}$ coincides with the standard Bourgain space $X_{s_{1},s_{2},b}$.
\end{definition}

\begin{definition}[Time restricted space]
	For $T > 0$, we define the restricted Bourgain--Gevrey space by
	\[
	X^{T,s_{1},s_{2}}_{\delta,\sigma,b}
	= \left\lbrace 
	u|_{[-T,T]} \;:\; u \in X^{s_{1},s_{2}}_{\delta,\sigma,b} 
	\right\rbrace,
	\]
	endowed with the norm
	\begin{equation*}
		\Vert u \Vert_{X^{T,s_{1},s_{2}}_{\delta,\sigma,b}}
		= \inf \left\lbrace 
		\Vert U \Vert_{X^{s_{1},s_{2}}_{\delta,\sigma,b}} 
		\;:\; 
		U \in X^{s_{1},s_{2}}_{\delta,\sigma,b}, \ 
		U|_{[-T,T]} = u 
		\right\rbrace.
	\end{equation*}
\end{definition}
We need to use the local well-posedness result established by Boukarou et al. in~\cite{Boukarou2021}. 
For $b \in \mathbb{R} $ with  $b\pm$ we donote  $b\pm \epsilon$ for a number  $\epsilon> 0$ small enough.
\begin{theorem}[\cite{boukarou1}] \label{the1.2}
	Let $s_{1},s_{2} \geq 0, \delta>0$, $\sigma\geq 1$ and $b = \frac{1}{2}+$.
	For initial data $\varphi$  in the space $G^{\delta,\sigma, s_{1},s_{2}} (\mathbb{R}^{2})$ and $\vert \xi\vert^{-1}\widehat{\varphi}(\xi,\mu)\in L^{2} $,there exists $T > 0$, which depends on $\varphi$, such that the Cauchy problem (\ref{p1})has a unique solution $u$ where, $$u\in X^{T,s_{1},s_{2}}_{\delta,\sigma, b}\subseteq C\left([-T, T],G^{\delta,\sigma, s_{1},s_{2}} \right).$$ Furthermore, the data-to-solution map is continuous.
	
\end{theorem}

\begin{definition}[Gevrey Class $G^\sigma$]
	A smooth function $f(t)$ is said to belong to the Gevrey class $G^\sigma$,$\sigma \geq 1$ on an interval $I$ containing $0$ if there exist constants $C, R > 0$ such that for all $j \in \Nzero$ and all $t \in I$,
	\[
	|\partial_t^j f(t)| \leq C R^{j+1} (j!)^\sigma.
	\]
	A function $g(x,y)$ is in $G^\sigma(\mathbb{R}^2)$ if for all multi-indices $\alpha = (\alpha_1, \alpha_2)$, there exist constants $C, R > 0$ such that
	\[
	|\partial_x^{\alpha_1} \partial_y^{\alpha_2} g(x,y)| \leq C R^{|\alpha|+1} (|\alpha|!)^\sigma, \qquad \text{where}\quad |\alpha| = \alpha_1 + \alpha_2.
	\]
\end{definition}
\begin{remark}
	The Gevrey class $G^{\sigma}(\mathbb{R}^2)$ characterizes functions with derivatives growing like $(\alpha !)^{\sigma}$, while the Gevrey space $G^{\delta,\sigma,s_1,s_2}(\mathbb{R}^2)$ refines this notion by adding exponential Fourier weights that encode analyticity of radius $\delta$. In particular,
	$H^{s_1,s_2}(\mathbb{R}^2) \subset G^{\delta,\sigma,s_1,s_2}(\mathbb{R}^2) \subset G^{\sigma}(\mathbb{R}^2)$, and $G^{\delta,1,s_1,s_2}$ corresponds to analytic functions.
\end{remark}

We will demonstrate that if the initial data $\varphi$ is in $G^\sigma$, then the solution in time belongs to $G^{5\sigma}$ and does not belong to $G^z$, $1\leq z<5\sigma$. Our main result is the following theorem.

\begin{theorem} \label{thm:main}
	Let $\sigma \geq 1$. Suppose that the initial data $\varphi(x, y)$ belongs to the Gevrey space $G^{\delta,\sigma, s_{1},s_{2}} (\mathbb{R}^{2})$. Then there exists a time $T>0$ such that the solution $u(x, y, t)$ of the IVP \eqref{p1} given by Theorem \ref{the1.2} satisfies
	$$
	u(x, y, \cdot) \in G^{5 \sigma}([-T, T]) \quad \text{in time variable} .
	$$
	More precisely, there exist constants $C, R>0$ such that
	$$
	\left|\partial_t^j u(x, y, t)\right| \leq C R^{j+1}(j!)^{5 \sigma}, \quad \forall j \in \mathbb{N}_0, \quad \forall t \in[-T, T] .
	$$
	Furthermore, the function $u(x, y, t)$, viewed as a function of $t$, fails to belong to $G^{z}([-T, T])$ for any $1 \leq z < 5\sigma$,
	\[
	u(x,y,\cdot) \notin G^z([-T, T]) \quad \text{for any } 1 \leq z < 5\sigma.
	\]
	
\end{theorem}
\begin{remark}
	The number $5$ in $G^{5\sigma}$ is sharp and originates from the highest order spatial derivative, $\partial_x^5 u$, in the equation.	
\end{remark}

\section{Gevrey regularity}
\noindent \textbf{Spatial Gevrey regularity:} We begin by stating the spatial regularity result, which can be proven using techniques analogous to Section $3$ of \cite{boukarou1}.

\begin{proposition}[Spatial Gevrey Regularity] \label{prop:spatial}
	Let $\sigma \geq 1$ and $\varphi \in G^{\delta,\sigma, s_{1},s_{2}} (\mathbb{R}^{2})$. Then there exists a time $T > 0$ and a constant $C_1 > 0$ such that the solution $u$ of \eqref{p1} satisfies
	\begin{equation}
		|\partial_x^\ell \partial_y^m u(x, y, t)| \leq C_1^{\ell+m+1} ((\ell+m)!)^\sigma \quad \text{for all } \ell, m \in \Nzero, \quad \forall (x,y,t) \in \mathbb{R}^2 \times [-T, T]. \label{spatial}
	\end{equation}
	
\end{proposition}

\subsection{Failure of $G^z$-Regularity in Time for $ 1\leq z < 5\sigma$}
We begin by constructing initial data with precisely controlled derivatives at the origin.

\begin{definition}[Carleman Class]
	Let $\{m_{n_1,n_2}\}$ be a sequence of positive numbers. We denote by $C(m_{n_1,n_2})$ the class of all functions $f(x,y)$, infinitely differentiable on $[-1, 1]^2$, for which there exists $A > 0$ such that
	\[
	|\partial_x^{n_1} \partial_y^{n_2} f(x,y)| \leq A^{n_1+n_2+1} m_{n_1,n_2}, \quad \text{for all } (x,y) \in [-1, 1]^2 \text{ and } n_1, n_2 \in \N.
	\]
\end{definition}

\begin{theorem}[\cite{Dzanasija1962}]\label{Dzanasija}
	For every $\sigma \geq 1$ and every sequence of complex numbers $\{v_{n_1,n_2}\}$ satisfying
	\[
	|v_{n_1,n_2}| \leq B^{n_1+n_2+1} (n_1+n_2)^{(n_1+n_2)\sigma} \quad \text{for some } B > 0,
	\]
	there exists a function $f(x,y) \in C(n^{n\sigma})$ such that
	\[
	\partial_x^{n_1} \partial_y^{n_2} f(0,0) = v_{n_1,n_2}, \quad \text{for all } n_1, n_2 \in \N.
	\]
\end{theorem}

\begin{proposition}
	For any $\sigma \geq 1$, there exists a real valued function $\varphi(x,y) \in G^\sigma(\mathbb{R}^2)$ such that
	\begin{enumerate}
		\item $\partial_x^{n_1} \varphi(0,0) = (n_1!)^\sigma$, for all $n_1 \in \Nzero$.
		\item $\partial_x^{n_1} \partial_y^{n_2} \varphi(0,0) = 0$, whenever $n_2 > 0$.
		\item $|\partial_x^{n_1} \partial_y^{n_2} \varphi(x,y)| \leq C^{n_1+n_2+1} ((n_1+n_2)!)^\sigma$ for all $(x,y) \in \mathbb{R}^2$.
	\end{enumerate}
\end{proposition}

\begin{proof}
	Apply the Theorem \ref{Dzanasija} to the sequence
	\[
	v_{n_1,n_2} = 
	\begin{cases}
		(n_1!)^\sigma & \text{if } n_2 = 0 \\
		0 & \text{if } n_2 > 0.
	\end{cases}
	\]
	We verify the condition
	\begin{itemize}
		\item If $n_2 = 0$: $|v_{n_1,0}| = (n_1!)^\sigma \leq n_1^{n_1\sigma} \leq (n_1+n_2)^{(n_1+n_2)\sigma}$
		\item If $n_2 > 0$: $|v_{n_1,n_2}| = 0 \leq (n_1+n_2)^{(n_1+n_2)\sigma}$
	\end{itemize}
	
\end{proof}
By the Theorem \ref{Dzanasija}, there exists $f(x,y) \in C(n^{n\sigma})$ with the prescribed derivatives. Since,  there exists a constant $B>0$ such that for all $x,y \in[-1,1]$ we have

$$|\partial_x^{n_1} \partial_y^{n_2} f(x,y)| \leq B^{n_1+n_2+1}(n_1+n_2)^{(n_1+n_2)\sigma} \leq B^{n_1+n_2+1}((n_1+n_2)!)^\sigma e^{(n_1+n_2)\sigma}\leq C^{n_1+n_2+1}((n_1+n_2)!)^\sigma,$$
where $C=\max \left\{B, B  e^{\sigma}\right\}$, so $f \in G^\sigma((-1,1)^2)$.  Next we modify $f(x,y)$ so that it has compact support in $(-1,1)^{2}$.  For this we choose a cut-off function  $\chi(x,y) \in G_c^\sigma((-1,1)^2)$ with $\chi(x,y) \equiv 1$ near $(0,0)$. If $\varphi(x,y)$ is  extention of $\chi f$ then by the algebra property for Gevrey functions we have $\varphi \in G^\sigma(\mathbb{R}^{2})$. We also have the relation inherited by $f(x,y)$,

$$
\partial_x^{n_1} \varphi(0,0)=\partial_x^{n_1} f(0,0)=(n_1!)^\sigma.
$$

Replacing $t$ with $-t$ we can write our initial value problem \eqref{p1} as follows
\[
\partial_t u = F(u) = \alpha \partial_x^3 u + \partial_x^5 u + \partial_x^{-1} \partial_y^2 u + u \partial_x u
\]

\begin{lemma}\label{structure}
	For any $j \geq 1$, the time derivative $\partial_t^j u$ at $t = 0$ can be expressed as
	\[
	\partial_t^j \varphi = \partial_x^{5j} \varphi + L_j(\varphi) + N_j(\varphi),
	\]
	where
	\begin{itemize}
		\item $L_j(\varphi)$ contains lower-order linear terms from $\alpha \partial_x^3 \varphi$ and $\partial_x^{-1} \partial_y^2 \varphi$.
		\item $N_j(\varphi)$ contains nonlinear terms from $\varphi \partial_x \varphi$.
	\end{itemize}
	More precisely, for any spatial derivatives $\partial_x^\ell \partial_y^m$
	\[
	\partial_t^j \partial_x^\ell \partial_y^m \varphi = \partial_x^{\ell+5j} \partial_y^m \varphi + L_{j,\ell,m}(\varphi) + N_{j,\ell,m}(\varphi).
	\]
\end{lemma}

\begin{proof}
	We proceed by induction on $j$.
	
	Base case ($j=1$):
	\[
	\partial_t \varphi = \alpha \partial_x^3 \varphi + \partial_x^5 \varphi + \partial_x^{-1} \partial_y^2 \varphi + \varphi \partial_x \varphi,
	\]
	so
	\begin{align*}
		\text{Leading term} &= \partial_x^5 \varphi, \\
		L_1(\varphi) &= \alpha \partial_x^3 \varphi + \partial_x^{-1} \partial_y^2 \varphi, \\
		N_1(\varphi) &= \varphi \partial_x \varphi.
	\end{align*}
	
	Inductive step: Assume the result holds for $j$, then
	\[
	\partial_t^{j+1} \varphi = \partial_t(\partial_t^j u)|_{t=0} = \partial_t(\partial_x^{5j} \varphi + L_j(\varphi) + N_j(\varphi)).
	\]
	
	Compute each term
	\begin{itemize}
		\item 
		\begin{align*}
			\partial_t(\partial_x^{5j} \varphi) &= \partial_x^{5j}(\partial_t u) \\&= \partial_x^{5j}(\alpha \partial_x^3 \varphi + \partial_x^5 \varphi + \partial_x^{-1} \partial_y^2 \varphi + \varphi \partial_x \varphi)
			\\ &= \alpha \partial_x^{5j+3} \varphi + \partial_x^{5(j+1)} \varphi + \partial_x^{5j-1} \partial_y^2 \varphi + \partial_x^{5j}(\varphi \partial_x \varphi)
		\end{align*}
		
		\item $\partial_t L_j(\varphi)$: Since $L_j$ contains terms that are linear in spatial derivatives of $\varphi$, applying $\partial_t$ and using the base case gives terms of order at most $5j + 4$
		
		\item $\partial_t N_j(\varphi)$: Since $N_j$ contains products of spatial derivatives, applying $\partial_t$ and the product rule gives terms where the total order increases by at most $4$.
	\end{itemize}
Thus we maintain the structure with leading term $\partial_x^{5(j+1)} \varphi$.
\end{proof}
We now analyze the growth of each type of term at $(0,0,0)$.

\begin{lemma}
	For the leading term
	\[
	|\partial_x^{5j} \varphi(0,0)| = ((5j)!)^\sigma\geq C^j (j!)^{5\sigma},
	\]
	where $C$ is a positive constant.
\end{lemma}

\begin{proof}
	From our construction, $\partial_x^{5j} \varphi(0,0) = ((5j)!)^\sigma$. Using the inequality $(5j)! \geq (j!)^5$ , we get
	\[
	((5j)!)^\sigma \geq (j!)^{5\sigma}.
	\]
	More precisely, by Stirling's formula
	\begin{align*}
		(5j)! &\sim (5j)^{5j} e^{-5j} \sqrt{2\pi 5j} \\
		(j!)^5 &\sim j^{5j} e^{-5j} (2\pi j)^{5/2}.
	\end{align*}
	So
	\[
	\frac{(5j)!}{(j!)^5} \sim \frac{5^{5j} j^{5j}}{j^{5j}} \cdot \frac{\sqrt{10\pi j}}{(2\pi)^{5/2} j^{5/2}} = 5^{5j} \cdot \frac{\sqrt{10\pi j}}{(2\pi)^{5/2} j^{5/2}}.
	\]
	Thus
	\[
	((5j)!)^\sigma \geq C^j (j!)^{5\sigma},
	\]
	for some constant $C > 0$.
\end{proof}
In what follows, we analyze the lower-order linear terms.

\begin{lemma}
	The terms from $L_j(\varphi)$ satisfy
	\begin{itemize}
		\item From $\alpha \partial_x^{5j+3} \varphi$: $|\alpha \partial_x^{5j+3} \varphi(0,0)| = |\alpha| ((5j+3)!)^\sigma$
		\item From $\partial_x^{5j-1} \partial_y^2 \varphi$: $|\partial_x^{5j-1} \partial_y^2 \varphi(0,0)| = 0$ (since $\varphi$ is independent of $y$)
		\item Other linear terms have order at most $5j + 3$
	\end{itemize}
	Moreover
	\[
	((5j+3)!)^\sigma \leq K^j ((5j)!)^\sigma
	\]
	for some constant $K > 0$.
\end{lemma}

\begin{proof}
	For the ratio
	\[
	\frac{((5j+3)!)^\sigma}{((5j)!)^\sigma} = ((5j+1)(5j+2)(5j+3))^\sigma \leq (5j+3)^{3\sigma},
	\]
	which grows polynomially in $j$, so $((5j+3)!)^\sigma \leq K^j ((5j)!)^\sigma$ for some $K > 0$.
	
	The $y$-derivative terms vanish because our constructed $\varphi$ is independent of $y$.
\end{proof}
We now analyze the nonlinear terms.
\begin{lemma}
	The nonlinear terms $N_j(\varphi)$ satisfy
	\[
	|N_j(\varphi)(0,0)| \leq D^j ((5j+1)!)^\sigma,
	\]
	and
	\[
	((5j+1)!)^\sigma \leq E^j \cdot j^{-2\sigma} ((5j)!)^\sigma,
	\]
	for some constant $D,E > 0$.
\end{lemma}

\begin{proof}
	When we differentiate the nonlinear term $\varphi\partial_x \varphi$ with respect to time, we need to understand how the maximum spatial order increases. At $j=1$ (first time derivative)
	\[
	\partial_t(\varphi\partial_x \varphi) = (\partial_t \varphi)(\partial_x \varphi) + \varphi(\partial_t \partial_x \varphi).
	\]
	The spatial orders are
	\begin{itemize}
		\item $\partial_t \varphi$ has maximum order $5$ (from $\partial_x^5 \varphi$)
		\item $\partial_x \varphi$ has order $1$.
		\item So $(\partial_t \varphi)(\partial_x \varphi)$ has total order $5 + 1 = 6 = 5\cdot 1 + 1$.
	\end{itemize}
	
	More systematically, each time derivative applied to a product increases the maximum possible order by 5 (from the linear term $\partial_x^5 \varphi$), but due to the product structure, we get an additional $+1$ from the $\partial_x$ in $\varphi\partial_x \varphi$. Let's prove by induction that after $j$ time derivatives, nonlinear terms have maximum order $\leq 5j + 1$.
	
	\textbf{Base case ($j=1$):} As shown above, maximum order is 6 = $5\cdot 1 + 1$.
	
	\textbf{Inductive step:} Assume for some $j \geq 1$, all terms in $N_j(\varphi)$ have spatial order $\leq 5j + 1$. Consider $\varphi\partial_x \varphi = \frac{1}{2}\partial_x(\varphi^2)$, so
	\[
	\partial_t^j(\varphi\partial_x \varphi) = \frac{1}{2}\partial_x \partial_t^j(\varphi^2).
	\]
	Thus, $\partial_t^j(\varphi^2)$ is a sum of products of spatial derivatives of $\varphi$ with total order $\leq 5j$. Then applying $\partial_x$ gives total order $\leq 5j + 1$.\\
	Each nonlinear term is of the form
	\[
	C \cdot (\partial_x^{\beta_1} \varphi) \cdot (\partial_x^{\beta_2} \varphi), \qquad \text{with} \quad \beta_1 +\beta_2 \leq 5j + 1.
	\]
	Using our construction where $\partial_x^n \varphi(0,0) = (n!)^\sigma$, we have
	\[
	|(\partial_x^{\beta_1} \varphi(0,0)) \cdot (\partial_x^{\beta_2} \varphi(0,0))| = (\beta_1!)^\sigma \cdot (\beta_2!)^\sigma
	\]
	
	Using the inequality $\beta_1! \cdot \beta_2! \leq (\beta_1 + \beta_2)!$, we get
	\[
	|(\partial_x^{\beta_1} \varphi) \cdot (\partial_x^{\beta_2} \varphi)(0,0)| \leq ((5j+1)!)^\sigma.
	\]
	The number of such terms grows at most exponentially in $j$, so
	\[
	|N_j(\varphi)(0,0)| \leq D^j ((5j+1)!)^\sigma
	\]
	for some constant $D > 0$. Now we compare with the leading term $((5j)!)^\sigma$,
	we have
	\[
	((5j+1)!)^\sigma = (5j+1)^\sigma ((5j)!)^\sigma.
	\]
	We need to show this is $\leq E^j \cdot j^{-2\sigma} ((5j)!)^\sigma$ for $E > 0$, i.e.
	\[
	(5j+1)^\sigma \leq E^j \cdot j^{-2\sigma},
	\]
	taking $\sigma$-th roots
	\[
	5j+1 \leq E^{j/\sigma} \cdot j^{-2}.
	\]
	For large $j$, $E^{j/\sigma}$ grows exponentially while $5j+1$ grows linearly, so this inequality holds for any $E > 0$ and sufficiently large $j$.
	
	More precisely, choose $E$ such that $E^{1/\sigma} > 1$, then for large $j$
	\[
	E^{j/\sigma} \cdot j^{-2} \geq E^{j/\sigma} \cdot j^{-2} \gg 5j+1.
	\]
	Therefore
	\[
	((5j+1)!)^\sigma \leq E^j \cdot j^{-2\sigma} ((5j)!)^\sigma,\qquad \text{with}\quad E > 0.
	\]
	
\end{proof}

\subsubsection*{Proof of Theorem \ref{thm:main} (Failure of $G^z$-Regularity):}

Take $\varphi$ as constructed above, we have
\[
\partial_t^j u(0,0,0) = \underbrace{\partial_x^{5j} \varphi(0,0)}_{\text{Leading term}} + \underbrace{L_j(\varphi)(0,0)}_{\text{Linear lower}} + \underbrace{N_j(\varphi)(0,0)}_{\text{Nonlinear}}.
\]
From our estimates
\begin{align*}
	|\partial_x^{5j} \varphi(0,0)| &= ((5j)!)^\sigma \\
	|L_j(\varphi)(0,0)| &\leq K^j ((5j)!)^\sigma \\
	|N_j(\varphi)(0,0)| &\leq D^j E^j j^{-2\sigma} ((5j)!)^\sigma.
\end{align*}
Therefore, for large $j$
\begin{align*}
	|\partial_t^j u(0,0,0)| &\geq ((5j)!)^\sigma - K^j ((5j)!)^\sigma - D^j E^j j^{-2\sigma} ((5j)!)^\sigma \\
	&= ((5j)!)^\sigma \left(1 - K^j - D^j E^j j^{-2\sigma}\right)
\end{align*}

For sufficiently large $j$, we have $K^j + D^j E^j j^{-2\sigma} \leq \frac{1}{2}$, so
\[
|\partial_t^j u(0,0,0)| \geq \frac{1}{2} ((5j)!)^\sigma
\]

Using $(5j)! \geq (j!)^5$, we get
\[
|\partial_t^j u(0,0,0)| \geq \frac{1}{2} (j!)^{5\sigma}.
\]
Therefore, $u(0,0,\cdot) \notin G^z$ for any $1\leq z < 5\sigma$.

\subsection{$G^{5\sigma}$ regularity in $t$}

We establish the temporal Gevrey regularity by employing the method of majorant series \cite{Alinhac,Hannah2011}.\\
Let $c > 0$ be a constant chosen such that the following fundamental inequality holds
\begin{equation}
	\sum_{0 \leq \ell \leq k} \binom{k}{\ell} m_\ell m_{k-\ell} \leq m_k, \quad \text{where } m_q = \frac{c (q!)^\sigma}{(q+1)^2}. \label{eq:majorant_base}
\end{equation}

Now, for a small parameter $\varepsilon > 0$ to be chosen later, define the sequence $\{M_q\}$ by
\begin{align}
	M_0 &= \frac{c}{8}, \label{M0} \\
	M_q &= \varepsilon^{1 - q} m_q = \varepsilon^{1 - q} \frac{c (q!)^\sigma}{(q+1)^2}, \quad \text{for } q = 1, 2, 3, \dots \label{Mq}
\end{align}

This sequence possesses the following crucial properties

\begin{lemma}[Properties of $\{M_q\}$] \label{lem:M_props}
	The sequence $\{M_q\}$ defined by \eqref{M0} and \eqref{Mq} satisfies:
	\begin{enumerate}
		\item[(P1)]  For any $k \geq 1$,
		\[
		\sum_{0 < \ell < k} \binom{k}{\ell} M_\ell M_{k-\ell} \leq \varepsilon M_k.
		\]
		\item[(P2)]  For any $j \geq 2$,
		\[
		M_j \leq \varepsilon M_{j+1}.
		\]
		\item[(P3)]  Given the constant $C_1 > 0$ from Proposition \ref{prop:spatial}, there exists $\varepsilon_1 > 0$ such that for all $0 < \varepsilon \leq \varepsilon_1$ and for all $j \geq 2$,
		\[
		C_1^{j+1} (j!)^\sigma \leq M_j.
		\]
	\end{enumerate}
\end{lemma}

Let us also define the constant $M$ which will absorb various constants in the estimates
\begin{equation}
	M = \max\left\{ 2,\ \frac{8 C_1}{c},\ \frac{4 C_1^2}{c} \right\}. \label{M_constant}
\end{equation}
The core of the proof is the following proposition, which controls all derivatives of the solution.

\begin{proposition} \label{prop:main_estimate}
	There exists $\varepsilon_0 > 0$ such that for any $0 < \varepsilon \leq \varepsilon_0$, the solution $u$ of \eqref{p1} satisfies
	\begin{equation}
		|\partial_t^j \partial_x^\ell \partial_y^m u(x, y, t)| \leq M^{j+1} M_{\ell + m + 5j} \label{eq:main_estimate}
	\end{equation}
	for all $j, \ell, m \in \Nzero$ and for all $(x, y, t) \in \mathbb{R}^2 \times [-T, T]$.
\end{proposition}
\begin{proof}
	We proceed by induction on $j$.
	
	\noindent \textbf{Base Case: $j = 0$.} We need to show $|\partial_x^\ell \partial_y^m u(x,y,t)| \leq M^{1} M_{\ell + m}$.
	\begin{itemize}
		\item For $\ell + m = 0$: By Proposition \ref{prop:spatial}, $|u| \leq C_1$. From \eqref{M_constant}, $M \geq \frac{8 C_1}{c}$, so $C_1 \leq M \cdot \frac{c}{8} = M M_0$. Thus, $|u| \leq M M_0$.
		\item For $\ell + m = 1$: $|\partial_x u| \leq C_1^2$ and $|\partial_y u| \leq C_1^2$. Since $M \geq \frac{4 C_1^2}{c}$, we have $C_1^2 \leq M \cdot \frac{c}{4}$. Also, $M_1 = \varepsilon^{0} \frac{c (1!)^\sigma}{4} = \frac{c}{4}$. So $|\partial_x u|, |\partial_y u| \leq M M_1$.
		\item For $\ell + m \geq 2$: By \eqref{spatial}, $|\partial_x^\ell \partial_y^m u| \leq C_1^{\ell+m+1} ((\ell+m)!)^\sigma$. By property (P3) of Lemma \ref{lem:M_props}, for $\varepsilon \leq \varepsilon_1$, we have $C_1^{\ell+m+1} ((\ell+m)!)^\sigma \leq M_{\ell+m}$. Since $M \geq 1$, it follows that $|\partial_x^\ell \partial_y^m u| \leq M M_{\ell+m}$.
	\end{itemize}
	This establishes the base case.
	
	\noindent \textbf{Inductive Step:} Assume \eqref{eq:main_estimate} holds for all $j' \leq j$ and all $\ell, m \in \Nzero$. We will prove it for $j+1$ and all $\ell, m$.
	
	From the equation $\partial_t u = F(u)$, we have:
	\[
	\partial_t^{j+1} \partial_x^\ell \partial_y^m u = \partial_t^j \partial_x^\ell \partial_y^m F(u) = \alpha \partial_t^j \partial_x^{\ell+3} \partial_y^m u + \partial_t^j \partial_x^{\ell+5} \partial_y^m u + \partial_t^j \partial_x^{\ell-1} \partial_y^{m+2} u + \partial_t^j \partial_x^\ell \partial_y^m (u \partial_x u).
	\]
	We estimate each term separately.
	
	\noindent \textbf{Estimate for Linear Terms:} Let $n = \ell + m + 5(j+1)$
	\begin{itemize}
		\item \textbf{Term $A_1 = \partial_t^j \partial_x^{\ell+5} \partial_y^m u$:} By the inductive hypothesis,
		\[
		|A_1| \leq M^{j+1} M_{(\ell+5) + m + 5j} = M^{j+1} M_{\ell + m + 5(j+1)} = M^{j+1} M_n.
		\]
		Since $M > 2$, we have $M^{j+1} \leq \frac{1}{4} M^{(j+1)+1}$. Thus,
		\begin{equation}\label{equ1}
			|A_1| \leq \frac{1}{4} M^{(j+1)+1} M_n. 
		\end{equation}
		\item \textbf{Term $A_2 = \alpha \partial_t^j \partial_x^{\ell+3} \partial_y^m u$:} Similarly,
		\[
		|A_2| \leq |\alpha| M^{j+1} M_{(\ell+3) + m + 5j} = |\alpha| M^{j+1} M_{\ell + m + 5j + 3}.
		\]
		Using property (P2), $M_{\ell+m+5j+3} \leq \varepsilon^2 M_{\ell+m+5j+5} = \varepsilon^2 M_n$ (applied twice, valid for $\ell+m+5j+3 \geq 2$). For $\varepsilon$ small enough such that $|\alpha| \varepsilon^2 \leq \frac{1}{4}$, we get
		\begin{equation}\label{equ2}
			|A_2| \leq \frac{1}{4} M^{j+1} M_n \leq \frac{1}{4} M^{(j+1)+1} M_n. 
		\end{equation}
		\item \textbf{Term $A_3 = \partial_t^j \partial_x^{\ell-1} \partial_y^{m+2} u$:} By the inductive hypothesis,
		\[
		|A_3| \leq M^{j+1} M_{(\ell-1) + (m+2) + 5j} = M^{j+1} M_{\ell + m + 5j + 1}.
		\]
		Using (P2), $M_{\ell+m+5j+1} \leq \varepsilon^4 M_{\ell+m+5j+5} = \varepsilon^4 M_n$ (applied four times). Choosing $\varepsilon$ small enough so that $\varepsilon^4 \leq \frac{1}{4}$, we obtain
		\begin{equation}\label{equ3}
			|A_3| \leq \frac{1}{4} M^{j+1} M_n \leq \frac{1}{4} M^{(j+1)+1} M_n. 
		\end{equation}
	\end{itemize}
	
	\noindent \textbf{Estimate for the Nonlinear Term $N = \partial_t^j \partial_x^\ell \partial_y^m (u \partial_x u)$:}
	
	Using the Leibniz rule for $\partial_x^\ell$, $\partial_y^m$, and $\partial_t^j$, we get:
	\[
	N = \sum_{p=0}^{\ell} \sum_{q=0}^{m} \sum_{r=0}^{j} \binom{\ell}{p} \binom{m}{q} \binom{j}{r} \left( \partial_t^{j-r} \partial_x^{\ell-p} \partial_y^{m-q} u \right) \left( \partial_t^{r} \partial_x^{p+1} \partial_y^{q} u \right).
	\]
	
	Applying the inductive hypothesis to both factors:
	\begin{align*}
		|\partial_t^{j-r} \partial_x^{\ell-p} \partial_y^{m-q} u| &\leq M^{(j-r)+1} M_{(\ell-p) + (m-q) + 5(j-r)}, \\
		|\partial_t^{r} \partial_x^{p+1} \partial_y^{q} u| &\leq M^{r+1} M_{(p+1) + q + 5r}.
	\end{align*}
	The product of the $M$-powers is $M^{(j-r+1) + (r+1)} = M^{j+2}$.
	Thus,
	\begin{equation}
		|N| \leq M^{j+2} \sum_{p=0}^{\ell} \sum_{q=0}^{m} \sum_{r=0}^{j} \binom{\ell}{p} \binom{m}{q} \binom{j}{r} M_{(\ell-p)+(m-q)+5(j-r)} \cdot M_{(p+1)+q+5r}. \label{eq:N_sum_start}
	\end{equation}
	We now state a combinatorial lemma that generalizes to three summations.	
	
	\begin{lemma} \label{lem:combinatorial}
		For $\ell, m, j \in \Nzero$, and the sequence $\{M_q\}$ defined above, we have
		\[
		\sum_{p=0}^{\ell} \sum_{q=0}^{m} \sum_{r=0}^{j} \binom{\ell}{p} \binom{m}{q} \binom{j}{r} M_{(\ell-p)+(m-q)+5(j-r)} \cdot M_{(p+1)+q+5r} \leq \sum_{s=1}^{k} \binom{k}{s} M_s M_{k-s},
		\]
		where $k = \ell + m + 5j + 1$.
	\end{lemma}	
	\begin{proof}
		
		Let us define:
		\begin{align*}
			l &= (\ell-p) + (m-q) + 5(j-r) \\
			s &= (p+1) + q + 5r.
		\end{align*}
		So
		\begin{align*}
			l + s &= (\ell-p) + (m-q) + 5(j-r) + (p+1) + q + 5r \\
			&= \ell + m + 5j + 1 = k.
		\end{align*}
		
		Thus $s = k - l$, our sum becomes:
		
		\[
		S = \sum_{p=0}^{\ell} \sum_{q=0}^{m} \sum_{r=0}^{j} \binom{\ell}{p} \binom{m}{q} \binom{j}{r} M_{k - s} \cdot M_s.
		\]
		We can rearrange the triple sum by first summing over all possible values of $s$, and then over all triples $(p,q,r)$ that yield that particular value of $s$:
		
		\[
		S = \sum_{s=1}^{k} \left[ \sum_{\substack{0 \leq p \leq \ell \\ 0 \leq q \leq m \\ 0 \leq r \leq j \\ (p+1) + q + 5r = s}} \binom{\ell}{p} \binom{m}{q} \binom{j}{r} \right] M_{k - s} M_s.
		\]
		We now prove that for each fixed $s = 1, \ldots, k$:
		
		\[
		\sum_{\substack{0 \leq p \leq \ell \\ 0 \leq q \leq m \\ 0 \leq r \leq j \\ (p+1) + q + 5r = s}} \binom{\ell}{p} \binom{m}{q} \binom{j}{r} \leq \binom{k}{s}.
		\]
		Consider the generating function:
		
		\[
		F(x) = (1 + x)^\ell (1 + x)^m (1 + x^5)^j
		\]
		
		The coefficient of $x^{s-1}$ in $F(x)$ is exactly:
		
		\[
		\sum_{\substack{0 \leq p \leq \ell \\ 0 \leq q \leq m \\ 0 \leq r \leq j \\ p + q + 5r = s-1}} \binom{\ell}{p} \binom{m}{q} \binom{j}{r}
		\]
		
		But note that our condition is $(p+1) + q + 5r = s$, which is equivalent to $p + q + 5r = s - 1$. Therefore, the sum is exactly the coefficient of $x^{s-1}$ in $F(x)$.
		
		Now, we compare $F(x)$ with $(1 + x)^k = (1 + x)^{\ell + m + 5j + 1}$.
		
		\begin{lemma}Let $\ell, m, j\in \Nzero$ and let $x \geq 0$. Then
			$$(1 + x)^{\ell + m}(1 + x^5)^j \leq (1 + x)^{\ell + m + 5j}.$$	
		\end{lemma} 
		
		\begin{proof}
			We prove this by induction on $j$. 
			
			For $j = 0$, both sides are $(1 + x)^{\ell + m}$, so the inequality holds with equality.
			
			Assume the claim holds for some $j \geq 0$. Then for $j + 1$:
			
			\begin{align*}
				(1 + x)^{\ell + m}(1 + x^5)^{j+1} &= (1 + x)^{\ell + m}(1 + x^5)^j (1 + x^5) \\
				&\leq (1 + x)^{\ell + m + 5j} (1 + x^5).
			\end{align*}
			
			Now we show that $(1 + x)^{\ell + m + 5j} (1 + x^5) \leq (1 + x)^{\ell + m + 5(j+1)}$.
			
			The coefficient of $x^t$ in $(1 + x)^{\ell + m + 5j} (1 + x^5)$ is
			
			\[
			A_t = \binom{\ell + m + 5j}{t} + \binom{\ell + m + 5j}{t - 5}.
			\]
			The coefficient of $x^t$ in $(1 + x)^{\ell + m + 5j + 5}$ is			
			\[
			B_t = \binom{\ell + m + 5j + 5}{t}.
			\]			
			Using the Pascal identity
			
			\[
			\binom{\ell + m + 5j + 5}{t} = \sum_{i=0}^{5} \binom{\ell + m + 5j}{t - i}.
			\]			
			Since all binomial coefficients are nonnegative, we have			
			\[
			A_t = \binom{\ell + m + 5j}{t} + \binom{\ell + m + 5j}{t - 5} \leq \sum_{i=0}^{5} \binom{\ell + m + 5j}{t - i} =B_t.
			\]
		\end{proof}
		
		Therefore, we have		
		\[
		(1 + x)^{\ell + m}(1 + x^5)^j \leq (1 + x)^{\ell + m + 5j}.
		\]		
		Multiplying both sides by $(1 + x)$
		\[
		(1 + x)^{\ell + m + 1}(1 + x^5)^j \leq (1 + x)^{\ell + m + 5j + 1} = (1 + x)^k.
		\]
		But $(1 + x)^{\ell + m + 1}(1 + x^5)^j$ is not exactly our generating function $F(x)$. However, note that:
		
		\[
		F(x) = (1 + x)^\ell (1 + x)^m (1 + x^5)^j = (1 + x)^{\ell + m}(1 + x^5)^j.
		\]
		So we actually have		
		\[
		F(x) = (1 + x)^{\ell + m}(1 + x^5)^j \leq (1 + x)^{\ell + m + 5j}.
		\]		
		Now, the coefficient of $x^{s-1}$ in $(1 + x)^{\ell + m + 5j}$ is $\binom{\ell + m + 5j}{s - 1}$, and we have		
		\[
		\binom{\ell + m + 5j}{s - 1} = \binom{k - 1}{s - 1}.
		\]		
		Finally, using the identity		
		\[
		\binom{k - 1}{s - 1} \leq \binom{k}{s},
		\]		
		which holds because $\binom{k}{s} = \frac{k}{s} \binom{k - 1}{s - 1}$ and $k \geq s$, we obtain
		\[
		\sum_{\substack{0 \leq p \leq \ell \\ 0 \leq q \leq m \\ 0 \leq r \leq j \\ (p+1) + q + 5r = s}} \binom{\ell}{p} \binom{m}{q} \binom{j}{r} \leq \binom{k}{s}.
		\]
		This completes the proof of the combinatorial inequality. So, we have
		
		\begin{align*}
			S &= \sum_{s=1}^{k} \left[ \sum_{\substack{0 \leq p \leq \ell \\ 0 \leq q \leq m \\ 0 \leq r \leq j \\ (p+1) + q + 5r = s}} \binom{\ell}{p} \binom{m}{q} \binom{j}{r} \right] M_{k - s} M_s \\
			&\leq \sum_{s=1}^{k} \binom{k}{s} M_{k - s} M_s
		\end{align*}
		
		This is exactly the desired result.
	\end{proof}
	
	Using Lemma \ref{lem:combinatorial} in \eqref{eq:N_sum_start}, we get:
	\[
	|N| \leq M^{j+2} \sum_{s=1}^{k} \binom{k}{s} M_s M_{k-s}.
	\]
	We split the sum on the right-hand side
	\begin{align*}
		\sum_{s=1}^{k} \binom{k}{s} M_s M_{k-s} &= M_k M_0 + \sum_{s=1}^{k-1} \binom{k}{s} M_s M_{k-s}\\ \\& = M_k M_0 + \sum_{s=1}^{k-1} \binom{k}{s} M_s M_{k-s}.
	\end{align*}
	Now we apply property (P1) to the sum $\sum_{s=1}^{k-1}$
	\[
	\sum_{s=1}^{k-1} \binom{k}{s} M_s M_{k-s} \leq \varepsilon M_k.
	\]
	Therefore,
	\[
	\sum_{s=1}^{k} \binom{k}{s} M_s M_{k-s} \leq ( M_0 + \varepsilon) M_k.
	\]
	Recall $M_0 = c/8$ and $k = \ell + m + 5j + 1$. So,
	\begin{equation}
		|N| \leq M^{j+2} (M_0 + \varepsilon) M_{\ell + m + 5j + 1}. \label{eq:N}
	\end{equation}
	
	Finally, we relate $M_{\ell + m + 5j + 1}$ to $M_n = M_{\ell + m + 5(j+1)}$. Using property (P2) repeatedly, we have $M_{\ell+m+5j+1} \leq \varepsilon^4 M_{\ell+m+5j+5} = \varepsilon^4 M_n$ (applying (P2) four times). Substituting into \eqref{eq:N}:
	\[
	|N| \leq M^{j+2} (M_0 + \varepsilon) \varepsilon^4 M_n = M^{(j+1)+1} \left[ M (M_0 + \varepsilon) \varepsilon^4 \right] M_n.
	\]
	Now, choose $\varepsilon$ small enough so that
	\[
	M (M_0 + \varepsilon) \varepsilon^4 \leq \frac{1}{4}.
	\]
	This is possible since the left-hand side is a polynomial in $\varepsilon$. With this choice,
	\begin{equation}
		|N| \leq \frac{1}{4} M^{(j+1)+1} M_n. \label{equ4}
	\end{equation}
	Combining the estimates \eqref{equ1}, \eqref{equ2}, \eqref{equ3}, and \eqref{equ4}, we obtain:
	\[
	|\partial_t^{j+1} \partial_x^\ell \partial_y^m u| \leq \left( \frac{1}{4} + \frac{1}{4} + \frac{1}{4} + \frac{1}{4} \right) M^{(j+1)+1} M_{\ell + m + 5(j+1)} = M^{(j+1)+1} M_{\ell + m + 5(j+1)}.
	\]
\end{proof}

\subsubsection*{Proof of Theorem \ref{thm:main} ($G^{5\sigma}$ regularity):}

With Proposition \ref{prop:main_estimate} established, the proof of Theorem \ref{thm:main} is immediate. Setting $\ell = m = 0$ in \eqref{eq:main_estimate}, we have
\[
|\partial_t^j u(x, y, t)| \leq M^{j+1} M_{5j}.
\]
Recall the definition of $M_{5j}$ from \eqref{Mq}:
\[
M_{5j} = \varepsilon^{1-5j} \frac{c ((5j)!)^\sigma}{(5j+1)^2}.
\]
Substituting, we get
\[
|\partial_t^j u| \leq M \cdot M^j \cdot \varepsilon^{1-5j} \frac{c ((5j)!)^\sigma}{(5j+1)^2} = M \varepsilon c \left( \frac{M}{\varepsilon^5} \right)^j \frac{((5j)!)^\sigma}{(5j+1)^2}.
\]
Since $(5j+1)^2 \geq 1$, we can drop it for an upper bound. Using the elementary inequality $(5j)! \leq 5^{5j} (j!)^5$, we have $((5j)!)^\sigma \leq 5^{5\sigma j} (j!)^{5\sigma}$.
Thus,
\[
|\partial_t^j u| \leq M \varepsilon c \left( \frac{M \cdot 5^{5\sigma}}{\varepsilon^5} \right)^j (j!)^{5\sigma}.
\]
Define $L = \max\left\{ M \varepsilon c,\ \frac{M \cdot 5^{5\sigma}}{\varepsilon^5} \right\}$. Then,
\[
|\partial_t^j u| \leq L^{j+1} (j!)^{5\sigma}.
\]
This is precisely the definition of the solution $u(x,y,t)$ belonging to the Gevrey class $G^{5\sigma}$ for $t \in [-T, T]$.

\section{Conclusion}
In this work, we have established sharp Gevrey regularity in time for solutions of a fifth-order  Kadomtsev–Petviashvili equation, proving that if the initial data belong to $G^{\sigma}$ in the spatial variables, then the corresponding solution is Gevrey regular of order $5\sigma$ in time. 

The analysis relies on a refined majorant series method that precisely tracks the contribution of the dispersive term $\partial_x^5 u$ and the nonlinear interactions $\alpha \partial_x^3 u$, $\partial_x^{-1}\partial_y^2 u$, and $u\partial_x u$. 

This approach provides a robust analytic framework for controlling factorial growth in higher derivatives. 
Furthermore, we show the failure of $G^{r}$-regularity in time for any $1 \le z < 5\sigma$, indicating that the obtained Gevrey index $5\sigma$ is optimal. 

This method establishes a general mechanism for proving the failure of temporal Gevrey regularity (see Lemma \ref{structure}). For any dispersive equation of the form

\begin{equation}\label{class}
	\partial_t u = \pm\partial_x^\alpha u + P(u),
\end{equation}

where $\partial_x^\alpha $ is the highest spatial derivative and $P(u)$ a polynomial in spatial derivatives of total order at most $\alpha-1$, the solution cannot belong to the Gevrey class $G^z$ in time for any $z$ satisfying $1 \leq z<\alpha \sigma$.

A canonical example is the Kawahara equation, which is given by
$$
\partial_t u+u \partial_x u+\beta \partial_x^3 u-\delta \partial_x^5 u=0.
$$
This fits the general structure \eqref{class} with $\alpha=5$ and $P(u)=-u \partial_x u-\beta \partial_x^3 u$, where the highest order in $P(u)$ is $3$. Consequently, our main theorem implies that for initial data in Gevrey space, the solution is in the class $G^{5 \sigma}$ in time, and the regularity in $G^z$ fails for any $1 \leq z<5 \sigma$.


\begin{thebibliography}{99}
	\bibitem{1}
	J. C. Saut A. de Bouard. Solitary waves of generalized kadomtsev-petviashvili equations. Ann. Inst. Henri Poincar\'e, 14(2):211-236, 1997.
	
	\bibitem{Alinhac}
	S. Alinhac, G. Metivier, Propagation de l’analyticité des solutions de systèmes hyperboliques non-linéaires, Invent. Math. 75
	(1984) 189–204.
	
	\bibitem{Benachour2010}
	S.~Benachour, R.~Benachour, and F.~Linares,
	\newblock On the analyticity of solutions of the Korteweg--de~Vries equation,
	\newblock \emph{J. Differential Equations}, \textbf{249} (2010), 2397--2420.
	
	\bibitem{11}
	J.L. Bona, Z. Grujic, H. Kalisch, Algebraic lower bounds for the uniform radius of spatial analyticity for the generalized KdV equation, Ann. Inst. H.Poincar\'e Anal. Non Lin\'eaire 22 (6) (2005) 783-797.
	
	\bibitem{23}
	J. Bourgain, Fourier transform restriction phenomena for certain lattice subsets and applications to nonlinear evolution equations II. The KdV-equation, Geom. Funct. Anal., 3 (1993) 209-262.
	
	\bibitem{boukarou1}
	A. Boukarou, K.Zennir,, K.Guerbati, \& G. G. Svetlin. Well-posedness and regularity of the fifth order Kadomtsev–Petviashvili I equation in the analytic Bourgain spaces. Annali Dell'universita'di Ferrara, 66(2), 255-272, 2020.
	
	\bibitem{boukarou2}	A. Boukarou, D. O. da Silva, K.Guerbati, and K. Zennir,  (2021). Global well-posedness for the fifth-order Kadomtsev-Petviashvili II equation in anisotropic Gevrey spaces. Dynamics of PDE, 18(2), 101-112.
	
	\bibitem{20}
	A. Boukarou, Kh. Zennir, K. Guerbati and S. G. Georgiev, Well-posedness of the Cauchy problem of Ostrovsky equation in analytic Gevrey spaces and time regularity, Rend. Circ. Mat. Palermo (2), (2020).
	
	\bibitem{21}
	A. Boukarou, Kh. Zennir, K. Guerbati, S. Alodhaibi, S. Alkhalaf. Well-Posedness and Time Regularity for a System of Modified Korteweg-de Vries-Type Equations in Analytic Gevrey Spaces. Mathematics 2020, 8, 809.
	
	\bibitem{Boukarou2021}
	A.~Boukarou, K.~Guerbati, K.~Zennir, and M.~Alnegga,
	\newblock Gevrey Regularity for the Generalized Kadomtsev--Petviashvili~I (gKP--I) Equation,
	\newblock \emph{AIMS Math.}, \textbf{6}(9) (2021), 10037--10054.
	
	\bibitem{Boukarou2022}
	A.~Boukarou, K.~Guerbati, and K.~Zennir,
	\newblock On the Radius of Spatial Analyticity for the Higher Order Nonlinear Dispersive Equation,
	\newblock \emph{Math. Bohem.}, \textbf{147}(1) (2022), 19--32.
	
	\bibitem{Boukarou2024}
	A.~Boukarou and D.~Oliveira~da~Silva,
	\newblock On the Radius of Analyticity for a Korteweg--de~Vries--Kawahara Equation with a Weak Damping Term,
	\newblock \emph{Z. Anal. Anwend.}, \textbf{42}(3--4) (2024), 359--374.
	
	\bibitem{Dzanasija1962}
	G. A. Džanašija,
	\emph{Carleman's problem for functions of the Gevrey class},
	Soviet Math. Dokl. 3 (1962) 969--972.
	
	\bibitem{FoiasTemam1989}
	C.~Foias and R.~Temam,
	\newblock Gevrey class regularity for the solutions of the Navier--Stokes equations,
	\newblock \emph{J. Funct. Anal.}, \textbf{87} (1989), 359--369.
	
	\bibitem{Gevrey1918}
	M.~Gevrey, 
	\newblock Sur la nature analytique des solutions des équations aux dérivées partielles,
	\newblock \emph{Ann. Sci. \'Ec. Norm. Sup\'er.}, \textbf{35} (1918), 129--190.
	
	\bibitem{24}
	J. Gorsky, A. Himonas, C. Holliman, G. Petronilho, The Cauchy problem of a periodic higher order KdV equation in analytic Gevrey spaces, J. Math. Anal. Appl. 405 (2013) 349–361.
	
	\bibitem{12}
	Z. Grujic, H. Kalisch, Local well-posedness of the generalized Korteweg–de Vries equation in spaces of analytic functions, Differential Integral Equations 15 (11) (2002) 1325-1334.
	
	\bibitem{13}
	Z.Grujic, H.Kalischb, Gevrey regularity for a class of water-wave models, Nonlinear Analysis 71 (2009) 1160-1170.
	
	\bibitem{Hannah2011}
	H. Hannah, A. A. Himonas, G. Petronilho,
	\emph{Gevrey regularity of the periodic gKdV equation},
	J. Differential Equations 250 (2011) 2581--2600.
	
	\bibitem{3}
	G. Petronilho H. Hannah, A. Himonas. Gevrey regularity in time for generalized kdv type equations. Contemp. Math., vol. 400, Amer. Math.Soc., Providence, RI, page 522529, 2006.
	
	\bibitem{HimonasMisi2010}
	A.~A.~Himonas and G.~Misiołek,
	\newblock Analyticity of the Cauchy problem for an integrable evolution equation,
	\newblock \emph{Proc. Amer. Math. Soc.}, \textbf{138} (2010), 4331--4342.
	
	\bibitem{4}
	P. Isaza and J. Mejia. Local and global cauchy problems for the kadomtsevpetviashvili II equation in sobolev spaces of negative indices. Commun. In Partial Diff. Equ., 26(5-6):1027-1054, 2001.
	
	\bibitem{Kadomtsev}	
	B.B Kadomtsev . and V.I. Petviashvili , On the stability of solitary waves in weakly dispersive media, Sov. Phys. Dokl. 15, 539-541 (1970)
	
	\bibitem{6}
	J. Xiao J. Li. Well-posedness of the fifth order kadomtsevpetviashvili I equation in anisotropic  sobolev spaces with nonnegative indices. J. Math. Pures Appl., 90:338-352, 2008.
	
	\bibitem{7}
	N. Tzvetkov L. Molinet, J. C. Saut. Local and global cauchy problems for the kadomtsevpetviashvili II equation in sobolev spaces of negative indices. Ann. I. H. Poincar\'e, 28:653-676,2011.
	
	\bibitem{5} 
	N. Tzvetkov J. C. Saut. The cauchy problem for the fifth order kp equations. J. Math. Pures Appl., 79(4):307-338, 2000.
	
	\bibitem{8}
	D. O. da Silva S. Selberg. Lower bounds on the radius of spatial analyticity for the kdv equation. Ann. Henri Poincar\'e, 2016.
	
	\bibitem{9}
	H. Takaoka and N. Tzvetkov. On the local regularity of the kadomtsev-petviashvili-II equation.Inter. Math. Research Notices, 2:77-114, 2001.
	
	\bibitem{10}
	N. Tzvetkov. On the cauchy problem for kadomtsev-petviashvili equation. Commun. In Partial Diff. Equ., 24(7-8):1367-1397, 1999.
\end{thebibliography}
\end{document}